\documentclass[10pt]{amsart}

\usepackage{amssymb, amsmath}
\usepackage{mathrsfs}
\usepackage{amscd}
\usepackage{verbatim}

\usepackage[colorlinks,linkcolor={blue},
citecolor={blue},urlcolor={red},]{hyperref}

\theoremstyle{plain}
\newtheorem{theorem}{Theorem}[section]
\theoremstyle{remark}
\newtheorem{remark}[theorem]{Remark}

\theoremstyle{plain}
\newtheorem{corollary}[theorem]{Corollary}
\newtheorem{lemma}[theorem]{Lemma}
\newtheorem{proposition}[theorem]{Proposition}

\numberwithin{equation}{section}

\newcommand{\E}{{\mathbb E}}
\renewcommand{\P}{{\mathbb P}}

\newcommand{\g}{\gamma}

\newcommand{\e}{\varepsilon}

\renewcommand{\O}{\Omega}

\newcommand{\lb}{\langle}
\newcommand{\rb}{\rangle}

\begin{document}

\title[Optimal probability lower bounds]{A note on optimal probability lower bounds for centered random variables}

\author{Mark Veraar}
\address{Institut f\"ur Analysis \\
Universit\"at Karlsruhe (TH)\\
D-76128  Karls\-ruhe\\Germany}

\email{mark@profsonline.nl}
\thanks{This work was carried out when the author was working in the Institute of
Mathematics of the Polish Academy of Sciences, supported by the
Research Training Network MRTN-CT-2004-511953}.



\date\today

\begin{abstract}
In this note we obtain lower bounds for $\P(\xi\geq 0)$ and $\P(\xi>
0)$ under assumptions on the moments of a centered random variable
$\xi$. The obtained estimates are shown to be optimal and improve
results from the literature. The results are applied to obtain
probability lower bounds for second order Rademacher chaos.
\end{abstract}

\maketitle
\section{Introduction}

In this note we obtain lower bounds for $\P(\xi\geq 0)$ and $\P(\xi>
0)$ under assumptions on the moments of $\xi$. Here $\xi$ is a
centered real-valued random variable. For instance we consider the
case where the first and $p$-th moment are fixed, and the case where
the second and $p$-th moment are fixed. Such lower bounds are used
in \cite{Bu,delaPG,HK,Kw87} to estimate tail probabilities. It can
be used to estimate $\P(\xi\leq E\xi)$ for certain random variables
$\xi$. Let $c_p = (\E|\xi|^p)^{\frac1p}$ and $c_{p,q}=c_p/c_q$.
Examples of known estimates that are often used for $p=2$ and $p=4$
are respectively
\[\P(\xi\geq 0)\geq \Big(\frac{c_{1,p}}{2}\Big)^{\frac{p}{p-1}} \ \ \ \text{and}  \ \ \
\P(\xi\geq 0)\geq \frac{1}{4 c_{p,2}^{\frac{2 p}{p-2}}}.\] A proof
of the first estimate can be found in \cite{delaPG}. The second
estimate is obtained in \cite{SQJS}. In this note we will improve
both estimates and in several cases we will show the obtained
results are sharp.

In the last part we give some applications of the results. We
improve an estimate for second order Rademacher chaos from
\cite{SQJS}. This result has applications to certain quadratic
optimization problems (cf. \cite{BNR,SQJS}). Finally, we give
applications to Hilbert space valued random variables. In particular
this improves a result from \cite{Bu}.

\section{Probability lower bounds}

The following result is an improvement of \cite[Proposition
3.3.7]{delaPG}.
\begin{proposition}\label{prop:case1p}
Let $\xi$ be a centered non-zero random variable and let $p\in (1,
\infty)$. Then
\begin{equation}\label{eq:case1p}
\P(\xi\geq 0)\geq \P(\xi> 0) \geq
\Big(\frac{c_{1,p}}{2}\Big)^{\frac{p}{p-1}}
(\psi^{-1}(c_{1,p}))^{-\frac{1}{p-1}}.
\end{equation}
Here $\psi:[\frac12, 1)\to (0,1]$ is the strictly decreasing
function defined by
%\[\frac{2^p}{c^p} \geq f(u):=\Big(\frac{1}{u^{\frac{1}{p-1}}}+ \frac{1}{(1-u)^{\frac{1}{p-1}}}\Big)^{p-1}\]
\[\psi(x) = 2 \Big( x^{-{\frac{1}{p-1}}} + (1-x)^{-\frac{1}{p-1}}
\Big)^{-\frac{p-1}{p}}.\]

The same lower bound holds for $\P(\xi< 0)$ and $\P(\xi\leq 0)$.
Moreover, the estimate \eqref{eq:case1p} for $\P(\xi\geq 0)$ and
$\P(\xi\leq 0)$ are sharp.

For all $p\in (1, \infty)$ the following bound holds
\begin{equation}\label{eq:approx}
\P(\xi\geq 0)\geq \P(\xi> 0) \geq
\Big(\frac{c_{1,p}}{2}\Big)^{\frac{p}{p-1}}
\Big(1-\Big(\Big(\frac{c_{1,p}}{2}\Big)^{-\frac{p}{p-1}}-1\Big)^{-(p-1)}\Big)^{-\frac{1}{p-1}}.
\end{equation}
\end{proposition}

The estimate \eqref{eq:case1p} improves the well-known estimate
$\P(\xi\geq 0)\geq \Big(\frac{c_{1,p}}{2}\Big)^{\frac{p}{p-1}}$ (cf.
\cite[Proposition 3.3.7]{delaPG}) by a factor
$(\psi^{-1}(c_{1,p}))^{-(p-1)}$. The lower bound \eqref{eq:approx}
is not optimal, but in general it is more explicit than
\eqref{eq:case1p}.

In the cases $p=2$ and $p=3$ one can calculate $\psi^{-1}$
explicitly. For $p=2$, the inverse is given by $\psi^{-1}(x) =
\frac12 + \frac12\sqrt{1-x^2}$. Therefore, a straightforward
calculation gives the following explicit lower bound, which is sharp
as well.
\begin{corollary}\label{cor:p=2}
Let $\xi$ be a centered non-zero random variable. Then
\[\P(\xi\geq 0)\geq \P(\xi> 0) \geq \frac12 -
\frac12\sqrt{1-c_{1,2}^2}.\]
\end{corollary}
This result can be used to slightly improve certain probability
lower bounds from \cite{HK}, where the estimate
$\frac{c_{1,2}^2}{4}$ is used.

\begin{proof}[Proof of Proposition \ref{prop:case1p}]
By symmetry we only need to consider $\P(\xi>0)$. By normalization
we may assume that $c_p=1$, and therefore $c= c_1=c_{1,p}$. Let $p_1
= \P(\xi>0)$ and $p_2 = \P(\xi<0)$. Let $\xi_+ = \max\{\xi, 0\}$ and
$\xi_- = \max\{-\xi, 0\}$. Then $0=\E\xi = \E\xi_+ - \E\xi_-$ and $c
= \E|\xi| = \E\xi_+ +\E\xi_-$. It follows that $\E\xi^+=
\E\xi^-=c/2$. Let $u=\E\xi_+^p$. Then $1-u = \E\xi_-^p$. By the
Cauchy-Schwartz inequality we have
\[c^p/2^p = (\E\xi_+)^p = (\E \xi_+ \text{sign}(\xi_+))^p \leq \E\xi_+^p (\E\text{sign}(\xi_+))^{p-1} = u p_1^{p-1}.\]
Therefore, $p_1\geq \Big(\frac{c^p}{2^p u}\Big)^{\frac{1}{p-1}}$.
Similarly, one can show $p_2\geq \Big(\frac{c^p}{2^p
(1-u)}\Big)^{\frac{1}{p-1}}$. It follows that
\[p_1 = 1-\P(\xi\leq 0) \leq 1-p_2 \leq 1- \Big(\frac{c^p}{2^p
(1-u)}\Big)^{\frac{1}{p-1}}.\] Therefore, to estimate $p_1$ from
below, we only need to consider the $u\in (0,1)$ which satisfy
\[\Big(\frac{c^p}{2^p u}\Big)^{\frac{1}{p-1}}\leq 1-\Big(\frac{c^p}{2^p
(1-u)}\Big)^{\frac{1}{p-1}}.\] This is equivalent with
\[\frac{2^p}{c^p} \geq f(u):=\Big(\frac{1}{u^{\frac{1}{p-1}}}+ \frac{1}{(1-u)^{\frac{1}{p-1}}}\Big)^{p-1}.\]
\[
c \leq \phi(u) = 2 \Big( u^{-{\frac{1}{p-1}}} +
(1-u)^{-\frac{1}{p-1}} \Big)^{-\frac{p-1}{p}}.
\]
Notice that $\phi$ is strictly increasing on $(0,\frac12]$ and
strictly decreasing on $[\frac12, 1)$. One easily checks that there
exists a unique $u_0\in (0,\frac12]$ and a unique $u_1\in [\frac12,
1)$ such that $\phi(u_0) = \phi(u_1) =c$. Moreover $c \leq \phi(u)$
holds if and only if $u\in [u_0, u_1]$. It follows that
$\Big(\frac{c^p}{2^p u}\Big)^{\frac{1}{p-1}}$ attains its minimum at
$u_1$, and therefore
\[p_1\geq \Big(\frac{c}{2}\Big)^{\frac{p}{p-1}} u_1^{-\frac{1}{p-1}}. %= \Big(\frac{c}{2 }\Big)^{\frac{p}{p-1}} (\psi^{-1}(c))^{-(p-1)}. %= \Big(\frac{1}{u_1 f(u_1)}\Big)^{\frac{1}{p-1}}.
\]
This completes the first part of the proof.

To prove \eqref{eq:approx}, note that it suffices to estimate
$\psi^{-1}$ from above, or equivalently $\psi$ from above. Clearly
for all $x\in [1/2, 1)$,
\[\psi(x)\leq 2 \Big(1+(1-x)^{-\frac{1}{p-1}} \Big)^{-\frac{p-1}{p}}=:\alpha(x).\]
Now $\alpha^{-1}(x)=1-\Big( \Big(\frac{x}{2}\Big)^{-\frac{p}{p-1}} -
1\Big)^{-(p-1)}$. This clearly implies the result.

To prove the sharpness of \eqref{eq:case1p} let $c\in (0,1]$ be
arbitrary and let $\mu=\Big(\frac{c}{2}\Big)^{\frac{p}{p-1}}
u_1^{-\frac{1}{p-1}}$, where $u_1 = \psi^{-1}(c)$. It suffices to
construct a centered random variable $\xi$ with $\E|\xi|^p=1$,
$\E|\xi| = c$ and $\P(\xi\leq 0)=\mu$. Let $x_1 = \frac{c}{2 \mu}$
and $x_2 = \frac{c}{2 (1-\mu)}$ and let $\xi=x_1$ with probability
$\mu$ and $\xi=x_2$ with probability $1-\mu$. Then $\E|\xi| = c$ and
\[\begin{aligned}
\E|\xi|^p &= \frac{c^p}{2^p}\big( \mu^{1-p}  +  (1-\mu)^{1-p}\big)
\\ & = \frac{c^p}{2^p}\Big( \frac{2^p}{c^p}u_1   +  \Big(1-   \big(\tfrac{c}{2}\big)^{\frac{p}{p-1}}u_1^{-\frac{1}{p-1}} \Big)^{1-p}\Big)
\\ & = \frac{c^p}{2^p}\Big( \frac{2^p}{c^p}u_1   +  \Big(\big(\tfrac{c}{2}\big)^{\frac{p}{p-1}}(1-u_1)^{-\frac{1}{p-1}} \Big)^{1-p}\Big)
\\ & = \frac{c^p}{2^p}\Big( \frac{2^p}{c^p}u_1   +  \frac{2^p}{c^p}
(1-u_1)\Big) = 1.
\end{aligned}\]
\end{proof}

In \cite{SQJS} it is shown that if $\xi$ satisfies $\E\xi=0$,
$E\xi^2=1$, $\E\xi^4\leq \tau$, then $\P(\xi\geq 0)$ and $\P(\xi\leq
0)$ are both greater or equal than $(2\sqrt{3}-3)/\tau$. Below we
will improve their result. More precisely we obtain sharp lower
bounds for $\P(\xi\leq 0), \P(\xi\geq 0)$, $\P(\xi<0)$ and
$\P(\xi>0)$.

\begin{proposition}\label{prop:less}
Let $\xi$ be a centered non-zero random variable. Then $\P(\xi\geq
0)\geq \P(\xi>0) \geq f(c_{4,2}^4)$, where
\begin{equation}
f(x) := \left\{%
\begin{array}{ll}
    \frac{1}{2} - \frac12\sqrt{\frac{x-1}{x+3}}, & \hbox{if $x\in [1, \frac{3\sqrt{3}}{2}-\frac32)$;} \\
    \frac{2\sqrt{3}-3}{x}, & \hbox{if $x\geq \frac{3\sqrt{3}}{2}-\frac32$.} \\
\end{array}%
\right.
\end{equation}
The same lower bound holds for $\P(\xi<0)$ and $\P(\xi\leq 0)$.
Moreover, the estimates are already sharp for $\P(\xi\geq 0)$ and
$\P(\xi\leq 0)$.
\end{proposition}
\begin{proof}
By symmetry we only need to consider $\P(\xi>0)$. By normalization
we may assume that $c_2=1$ and therefore $c:=c_{4}^4=c_{4,2}^4$. The
proof of the first part is a slight modification of the argument in
\cite{SQJS}. Let $p_1 = \P(\xi>0)$ and $p_2 = \P(\xi<0)$. Let $\xi_+
= \max\{\xi, 0\}$ and $\xi_- = \max\{-\xi, 0\}$. Then $0=\E\xi =
\E\xi_+ - \E\xi_-$. Let $s=\E\xi^+= \E\xi^-$. By H\"older's
inequality we have $\E\xi_+^2 \leq (\E\xi^4_+)^{\frac13}
s^{\frac23}$ and $\E\xi_-^2 \leq (\E\xi^4_-)^{\frac13} s^{\frac23}$.
From this and $1=\E\xi^2 = \E\xi_+^2+\E\xi_-^2$ we obtain that
\[c \geq \E\xi_+^4 + \E\xi_-^4 \geq (\E\xi_+^2)^{3} s^{-2} + (\E\xi_-^2)^{3} s^{-2} = (u^3 +(1-u)^3)s^{-2},\]
where $u=\E\xi^2_+$. On the other hand by the Cauchy-Schwartz
inequality we have
\[s^2 = (\E\xi_+)^2 = (\E \xi_+ \text{sign}(\xi_+))^2 \leq \E\xi_+^2 (\E\text{sign}(\xi_+)) = u p_1.\]
Therefore, $p_1\geq \frac{u^3+ (1-u)^3}{u c}$. Minimization over
$u\in (0,1)$ gives $u=\frac{1}{\sqrt{3}}$ and $p_1\geq
\frac{(2\sqrt{3}-3)}{c}$.

Next we improve the estimate for $c\in [1,
\frac{3\sqrt{3}}{2}-\frac32)$. In the same way as for $p_1$, one can
show that $p_2 \geq \frac{u^3+ (1-u)^3}{(1-u) c}$. Therefore,
\[p_1 = 1-\P(\xi<0)\leq 1-p_2 \leq \frac{u^3+ (1-u)^3}{(1-u) c}. \]
Combining this with the lower estimate for $p_1$, the only $u\in
(0,1)$ which have to be considered are those for which
\[\frac{u^3+ (1-u)^3}{u c} \leq 1-\frac{u^3+ (1-u)^3}{(1-u)c}.\]
One easily checks that this happens if and only if
\[u_0=\frac{1}{2} - \frac12\sqrt{\frac{c-1}{c+3}}\leq u\leq \frac{1}{2} + \frac12\sqrt{\frac{c-1}{c+3}}=u_1.\]
For the $c$'s we consider one may check that
$\frac{1}{\sqrt{3}}\notin (u_0, u_1)$. Therefore, the minimum is
attained at the boundary. Since $g(u_0) = u_1$ and $g(u_1) = u_0$,
$u_0$ is the minimum of $g$ on $[u_0, u_1]$. This shows that
$p_1\geq u_0$.

To show this estimate is sharp for $x\geq
\frac{3\sqrt{3}}{2}-\frac32$ we will construct a certain family of
random variables $(\xi_\e)_{\e\geq 0}$. Let $\e\geq 0$ be not too
large. Let $\xi_{\e}$ be equal to $x_i(\e)$ with probability
$\lambda_i$, for $i=1, 2, 3$. Let
\[
\lambda_1 =  \Big(\frac32 - \frac{\sqrt{3}}{2}\Big)/c,  \ \lambda_2
= 1- \Big(\frac{3\sqrt{3}}{2}- \frac32\Big)/c,  \ \lambda_3 =
(2\sqrt{3}-3)/c.
\]
Let $x_2(\e) = -\e$, and let $x_1(\e)<0$ and
$x_3(\e)>0$ be the solution of
\[\E\xi=\lambda_1 x_1 + \lambda_2 \e +\lambda_3 x_3 = 0\]
\[\E\xi^2=\lambda_1 x_1^2 + \lambda_2 \e^2 +\lambda_3 x_3^2 = 1.\]
Notice that
\[x_1(0) = -\frac{1-\frac13 \sqrt{3} }{\sqrt{2-\sqrt{3}}}\sqrt{c}, \ x_2=0, \ \ x_3(0) = \frac{\frac13 \sqrt{3} }{\sqrt{2-\sqrt{3}}} \sqrt{c}.\]
For $\e>0$ small enough one may check that
$x_1(\e)<x_2(\e)<0<x_3(\e)$, and $P(\xi_\e\geq 0) = \lambda_3$.
Moreover, it holds that
\[\lim_{\e\downarrow 0}\E\xi^4_{\e} = \lim_{\e\downarrow 0}\lambda_1 x_1^4(\e) + \lambda_2 x_2^4(\e) +\lambda_3 x_3^4(\e) = \lambda_1 x_1^4(0) + \lambda_2 x_2^4(0) +\lambda_3 x_3^4(0) = c.\]
This completes the proof.

The sharpness of the result for $x\in [1,
\frac{3\sqrt{3}}{2}-\frac32)$ follows if we take $\xi$ a random
variable with two values. Indeed, let $x_2 =
\frac12\sqrt{2+2c+2\sqrt{(c-1)(c+3)}}$, $x_1=-1/x_2$, $\lambda_1 =
x_2/(x_2-x_1)$ and $\lambda_2=-x_1/(x_2-x_1)$. One easily checks
that $\E\xi=0$, $\E\xi^2=1$ and $\E\xi^4=c$ and
\[\lambda_1= \frac{1}{2} - \frac12\sqrt{\frac{c-1}{c+3}}.\]
\end{proof}

In \cite{SQJS} also a lower bound is obtained if one uses the $p$-th
moment instead of the fourth moment. They show that $\P(\xi\geq
0)\geq \frac{1}{4} c_{p,2}^{-\frac{2p}{p-2}}$. In the next remark we
improve the factor $\frac14$.

\begin{remark}
Let $\xi$ be a centered non-zero random variable and let $p\in (2,
\infty)$. Then
\[\P(\xi\geq 0)\geq \P(\xi>0) \geq \frac{1}{4} c_{p,2}^{-\frac{2p}{p-2}} \Big({(3-4/p)^{-\frac{1}{p-2}}}+1\Big) \geq \frac{(e^{-1}+1)}{4}c_{p,2}^{-\frac{2p}{p-2}}.\]
\end{remark}
\begin{proof}
It follows from the proof in \cite{SQJS} that $\P(\xi>0) \geq
\min_{u\in (0,1)} c_{p,2}^{-\frac{2p}{p-2}}f(u)$, where $f(u) =
\frac{1}{u} (u^{p-1} + (1-u)^{p-1})^{\frac{2}{p-2}}$. The function
$f$ has a minimum $u=u_0$ in $[\tfrac12,1)$. Moreover it satisfies
$f'(u_0)=0$.

Indeed, if $u_0\in (0,\tfrac12)$ would be a minimum of $f$ then,
$f(1-u_0)<f(u_0)$, which is impossible. That a minimum $u$ exists on
$[\frac12,1)$ and that it satisfies $f'(u)=0$ is clear. A
calculation shows that $f'(u) = \alpha(u) g(u)$, where $\alpha(u)>0$
and
\[g(u) = p u^{p-1} -p (1-u)^{p-2} u - p (1-u)^{p-2} + 2 (1-u)^{p-2}.\]
Therefore, $f' (u) = 0$ if and only if $g(u) = 0$. Let us estimate
$u_0$ from above. Since $g(u_0)=0$, we have
\[(1-u_0)^{p-2} \Big(1-\frac{2}{p}\Big) = u_0(u_0^{p-2}-(1-u_0)^{p-2}).\]
Using that $u_0\geq \frac12$, we obtain that
\[(1-u_0)^{p-2} \Big(1-\frac2p\Big) \geq \frac12(u_0^{p-2}-(1-u_0)^{p-2}),\]
and therefore
\[\frac{1}{u_0} \geq {(3-4/p)^{-\frac{1}{p-2}}}+1.\]

We conclude that
\[f(u) \geq \Big({(3-4/p)^{-\frac{1}{p-2}}}+1\Big) (u^{p-1} + (1-u)^{p-1})^{\frac{2}{p-2}}\geq \Big({(3-4/p)^{-\frac{1}{p-2}}}+1\Big) \frac{1}{4}.\]
The final estimate follows from $(3-4/p)^{\frac{1}{p-2}}\downarrow
e$ as $p\downarrow 2$.
\end{proof}

\section{Applications}

We will need the following estimate for second order chaoses. It is
well-known to experts. For a random variable $\xi$ and $p\in [1,
\infty)$, let $\|\xi\|_p = (\E|\xi|^p)^{\frac1p}$.

\begin{lemma}\label{lem:momentineq}
Let $(\xi_{i})_{i\geq 1}$ be an i.i.d. sequence of symmetric random variables
with $\E|\xi_i|^2= 1$ and $\E|\xi_i|^4\leq 3$. Then for any real numbers
$(a_{i,j})_{1\leq i<j\leq n}$ it holds that
\begin{equation}\label{eq:rad42}
\Big\|\sum_{1\leq i<j\leq n} \xi_i \xi_j a_{ij}\Big\|_4 \leq
\sqrt[4]{15} \Big\|\sum_{1\leq i<j\leq n} \xi_i \xi_j
a_{ij}\Big\|_2.
\end{equation}
Moreover, in the case $(\xi_i)_{i\geq 1}$ is a Rademacher sequence
or a Gaussian sequences the inequality \eqref{eq:rad42} is sharp.
\end{lemma}
\begin{proof}
For $j>i$ let $a_{ij} = a_{ji}$ and let $a_{ii}=0$. By homogeneity
we may assume that
\begin{equation}\label{eq:ass12}
\Big\|\sum_{1\leq i<j\leq n} \xi_i \xi_j a_{ij}\Big\|_2^2 =
\sum_{1\leq i<j\leq n} a_{ij}^2=\frac12.
\end{equation} Let
$(\g_i)_{i\geq 1}$ be a sequence of independent standard Gaussian
random variables. Since $\E|\xi_i|^2 \leq \E|\g_i|^2$ and
$\E|\xi_i|^4 \leq \E|\g_i|^4$, we have that
\begin{equation}\label{eq:radGaus}
\Big\|\sum_{1\leq i<j\leq n} \xi_i \xi_j a_{ij}\Big\|_4\leq
\Big\|\sum_{1\leq i<j\leq n} \g_i \g_j a_{ij}\Big\|_4
\end{equation}
Denote by $A$ the matrix $(a_{ij})_{1\leq i,j\leq n}$. By
diagonalization we may write $A = PDP^T$, where $D=(\lambda_i)$ is a
diagonal matrix and $P$ is an orthogonal matrix. Clearly, $\lb
A\gamma,\gamma \rb = \lb D\gamma',\gamma' \rb$, where $\gamma=\g_1,
\ldots, \g_n$ and $\gamma'= P^T \g$. Since $P$ is orthogonal
$\gamma'$ has the same distribution as $\g$. Therefore,
\[0=\E\lb A\gamma,\gamma \rb = \E\lb D\gamma',\gamma' \rb
= \sum_{i=1}^n \lambda_i.\] Similarly one may check that
%\[2 = \E\lb A\gamma,\gamma \rb^2 = \E|\lb D\gamma',\gamma' \rb|^2 = \E\Big|\sum_{i=1}^n \lambda_i \g_i^2\Big|^2 = \E\Big|\sum_{i=1}^n \lambda_i (\g_i^2-1)\Big|^2 = \sum_{i=1}^n \lambda_i^2 \E(\g_i^2-1)^2 = 2\sum_{i=1}^n \lambda_i^2.\]
$\sum_{i=1}^n \lambda_i^2 = 1$. It follows that
\[\begin{aligned}
\E\lb A\gamma,\gamma \rb^4 & = \E|\lb D\gamma',\gamma' \rb|^4 =
\E\Big|\sum_{i=1}^n \lambda_i (\g_i^2-1)\Big|^4
%\\ & = \sum_{i=1}^n \lambda_i^4 \E(\g_i^2-1)^4 + 6 \sum_{i=1}^n \sum_{j=1, j\neq i}^n \lambda_i^2 \lambda_j^2 \big(\E(\g_i^2-1)^2\big)^2
% \\ & = 60\sum_{i=1}^n \lambda_i^4  + 24 \sum_{i=1}^n \lambda_i^2 (1-\lambda_i^2)
=36 \sum_{i=1}^n \lambda_i^4  + 24 \sum_{i=1}^n \lambda_i^2
\\ & \leq 36\Big(\sum_{i=1}^n \lambda_i^2\Big)^{2}  + 24 \sum_{i=1}^n \lambda_i^2 = 60.
\end{aligned}\]
Therefore,
\[\E\Big|\sum_{1\leq i<j\leq n} \g_i \g_j a_{ij}\Big|^4 = \frac{1}{16} \E\lb A\gamma,\gamma \rb^4 \leq \frac{15}{4}.\]
Recalling \eqref{eq:ass12} and \eqref{eq:radGaus} this implies the
result.

To show that the inequality \eqref{eq:rad42} is sharp it suffices to
consider the case where the $(\xi_i)_{i\geq 1}$ are standard
Gaussian random variables. Indeed, if \eqref{eq:rad42} holds for a
Rademacher sequence $(\xi_i)_{i\geq 1}$, then the central limit
theorem implies \eqref{eq:rad42} for the Gaussian case. Now assume
$(\xi_i)_{i\geq 1}$ are standard Gaussian random variables. Let
$a_{ij} = 1$ for all $i\neq j$ and $a_{ii}=0$. Notice that
$\sum_{1\leq i<j\leq n} \xi_i \xi_j a_{ij} = \frac12\lb A
\xi,\xi\rb$, where $\xi=(\xi_{i})_{i=1}^n$. For the right-hand side
of \eqref{eq:rad42} we have
\[\Big\|\sum_{1\leq i<j\leq n} \xi_i \xi_j a_{ij}\Big\|_2^2 =\sum_{1\leq i<j\leq n} a_{ij}^2 = \frac{n(n-1)}{2}.\]
As before, we may write $A = P D P^T$, where $D$ is the diagonal
matrix with eigenvalues $(\lambda_i)_{i=1}^n$ of $A$ and $P$ is
orthogonal. It is easy to see that the eigenvalues of $A$ are $n-1$
and $-1$, where the latter has multiplicity $n-1$. By the same
calculation as before it follows that
\[\E\lb A \xi,\xi\rb^4 = 60\sum_{i=1}^n \lambda_i^4  + 24 \sum_{i\neq j}^n \lambda_i^2 \lambda_j^2 = 36 ((n-1)^4 +n) + 24 ((n-1)^2+n)^2.\]
Letting $C$ denote the best constant in \eqref{eq:rad42} gives that
\[\frac{36}{16} ((n-1)^4 +n) + \frac{24}{16} ((n-1)^2+n)^2 \leq C^4 \frac{n^2(n-1)^2}{4}.\]
Dividing by $n^4/4$ and letting $n$ tend to infinity yields $9+6
\leq C^4$, as required.
\end{proof}

By standard arguments (cf. \cite[Chapter 3]{delaPG}) using
H\"older's inequality one also obtains from Lemma
\ref{lem:momentineq} that
\begin{equation}\label{eq:radp2}
\Big\|\sum_{1\leq i<j\leq n} \xi_i \xi_j a_{ij}\Big\|_p \leq
15^{\frac{p-2}{2p}} \Big\|\sum_{1\leq i<j\leq n} \xi_i \xi_j
a_{ij}\Big\|_2, \ \ \text{for} \ p\in (2,4)
\end{equation}
and
\begin{equation}\label{eq:rad2p}
\Big\|\sum_{1\leq i<j\leq n} \xi_i \xi_j a_{ij}\Big\|_2 \leq
15^{\frac{2-p}{2p}} \Big\|\sum_{1\leq i<j\leq n} \xi_i \xi_j
a_{ij}\Big\|_p, \ \ \text{for} \ p\in (0,2).
\end{equation}

As an immediate consequence of Proposition \ref{prop:less} and Lemma
\ref{lem:momentineq} we obtain the following result. We state it for
Rademacher random variables, but the same result holds for random
variables $(\xi_n)_{n\geq 1}$ as in Lemma \ref{lem:momentineq}.
\begin{proposition}\label{thm:main}
Let $(r_i)_{i\geq 1}$ be a Rademacher sequence. For any real numbers
$(a_{ij})_{i,j=1}^n$ it holds that
\[\P\Big(\sum_{1\leq i<j\leq n} r_i r_j a_{ij}\geq 0\Big) \geq \frac{2\sqrt{3}-3}{15}> \frac{3}{100}.\]
If not all $a_{ij}$ are identically zero then
\[\P\Big(\sum_{1\leq i<j\leq n} r_i r_j a_{ij}> 0\Big) \geq \frac{2\sqrt{3}-3}{15}> \frac{3}{100}.\]
\end{proposition}
This result has applications to certain quadratic optimization
problems (cf. \cite{BNR} and \cite[Theorem 4.2]{SQJS}). It improves
the known result with $\frac{1}{87}$ from \cite[Lemma 4.1]{SQJS}.

The conjecture (see \cite{BNR}) is that the estimate in Proposition
\ref{thm:main} holds with $\frac14$. The methods we have described
will probably never give such a bound, and a more sophisticated
argument will be needed. However, another conjecture is that for a
Rademacher sequence $(r_i)_{i\geq 1}$ and $p=1$, \eqref{eq:rad2p}
holds with constant $2$, i.e.
\[
\Big\|\sum_{1\leq i<j\leq n} r_i r_j a_{ij}\Big\|_2 \leq
2\Big\|\sum_{1\leq i<j\leq n} r_i r_j a_{ij}\Big\|_1.
\]
If this would be true, then Corollary \ref{cor:p=2} implies that
\[\P\Big(\sum_{1\leq i<j\leq n} r_i r_j a_{ij}\geq 0\Big) \geq \frac12 - \frac14\sqrt{3} > \frac{1}{15}\]
which is better than $\frac{3}{100}$.

\begin{remark}
Let $(\eta_i)_{i\geq 1}$ be independent exponentially distributed random
variables with $\E \eta_i = 1$ and let $\xi = \sum_{i=1}^n a_i(\eta_i-1)$ for
real numbers $(a_{i})_{i\geq 1}$. In \cite{SQJS} the estimate $\P(\xi\geq
0)>\frac{1}{20}$ has been obtained. This follows from Proposition
\ref{prop:less} and (see \cite{SQJS})
\begin{equation}\label{eq:exp42}
(\E|\xi|^4)^{\frac14}\leq 9 (\E|\xi|^2)^{1/2}.
\end{equation}
The inequality \eqref{eq:exp42} is optimal. As in \eqref{eq:rad2p}
we have that \eqref{eq:exp42} implies that
\[(\E|\xi|^2)^{\frac12}\leq  C \E|\xi|)\]
for a certain constant $C$ and $C\leq 3$. One the other hand, taking $n=2$, and
$a_1=1$, $a_2=-1$, gives that $C\geq \sqrt{2}$. It is interesting to find the
optimal value of $C$. If this value is small enough, then Proposition
\ref{prop:case1p} will give a better result than $\frac{1}{20}$.

A similar situation can be considered if one replaces $\eta_i$ by
$\g_i^2$.
\end{remark}

Next we prove another probability bound. A uniform bound can already
be found in \cite{Bu}.
\begin{corollary}
Let $(r_i)_{i\geq 1}$ be a Rademacher sequence. Let $(H, \lb\cdot,
\cdot\rb)$ be a Hilbert space. For any vectors $(a_i)_{i=1}^n$ from
$H$ it holds that
\begin{equation}\label{eq:radlower}
\P\Big(\Big\|\sum_{i=1}^n r_i a_i\Big\|\leq \Big(\sum_{i=1}^n
\|a_i\|^2\Big)^{\frac12} \Big)\geq \frac{2\sqrt{3}-3}{15}>
\frac{3}{100},
\end{equation}
\begin{equation}\label{eq:radupper}
\P\Big(\Big\|\sum_{i=1}^n r_i a_i\Big\|\geq \Big(\sum_{i=1}^n
\|a_i\|^2\Big)^{\frac12} \Big)\geq
\frac{2\sqrt{3}-3}{15}>\frac{3}{100}.
\end{equation}
\end{corollary}
For real numbers $(a_i)_{i=1}^n$, \eqref{eq:radlower} holds with
constant $\frac{3}{8}$ (see \cite{HoKl}). The well-known conjecture
is that it holds with $\frac12$. Again for real numbers
$(a_i)_{i=1}^n$ \eqref{eq:radupper} holds with constant
$\frac{1}{10}$ (see \cite{Ol}). The conjecture (see \cite{HK}) is
that it holds with constant $\frac{7}{64}$.
\begin{proof}
As in \cite{Bu} one can show that
\[\P\Big(\Big\|\sum_{i=1}^n r_i a_i\Big\|\geq \Big(\sum_{i=1}^n \|a_i\|^2\Big)^{\frac12}
\Big) = \P\Big(\sum_{1\leq i<j\leq n} r_i r_j a_{ij}\geq 0\Big),\]
where $a_{ij} = 2\text{Re}(\lb a_i, a_j\rb)$. Therefore, the result
follows from Proposition \ref{thm:main}. The proof of
\eqref{eq:radlower} is the same.
\end{proof}

In the next result we obtain a probability bound for Gaussian random
variables with values in a Hilbert space.
\begin{proposition}
Let $H$ be a real separable Hilbert space and let $G:\O\to H$ be a
nonzero centered Gaussian random variable. Then
\begin{equation}\label{eq:GaussianH}
\frac{2\sqrt{3}-3}{15} \leq \P(\|G\|> (\E\|G\|^2)^{\frac12}) \leq
\frac12.
\end{equation}
\end{proposition}
By \cite{Kw94} the upper bound $\frac12$ is actually valid for
Gaussian random variables with values in a real separable Banach
space. We also refer to \cite{SB} for related results on Gaussian
quadratic forms.
\begin{proof}
It is well-known that we can find independent standard Gaussian
random variables $(\g_n)_{n\geq 1}$, orthonormal vectors
$(a_n)_{n\geq 1}$ in $H$ and positive numbers $(\lambda_n)_{n\geq
1}$ such that $G = \sum_{n\geq 1} \sqrt{\lambda_n} \g_n a_n$, where
the series converges almost surely in $H$. The convergence also
holds in $L^2(\O;H)$. Notice that
\[\xi:=\|G\|^2 - \E\|G\|^2 = \sum_{n\geq 1}\lambda_k(\g_k^2-1),\]
so that as in Lemma \ref{lem:momentineq} $\E\xi^2 = 2\sum_{n\geq 1}
\lambda^2_k$ and $\E\xi^4 \leq 60 \sum_{n\geq 1} \lambda^2_k$.
Therefore the lower estimate follows from Proposition
\ref{prop:less}.
\end{proof}

{\em Acknowledgment} -- The author thanks professor S. Kwapie\'n for
helpful discussions.

\providecommand{\bysame}{\leavevmode\hbox
to3em{\hrulefill}\thinspace}

\end{document}